\documentclass[]{article}

\usepackage[]{geometry}

\usepackage{amsmath,amsthm} 
\usepackage{amssymb}  
\usepackage{amsfonts}
\usepackage{graphicx}
\usepackage{epstopdf}
\usepackage{mathtools}

\DeclarePairedDelimiter\floor{\lfloor}{\rfloor}
\usepackage{color,hyperref}
\usepackage{cite}
\makeatletter

\newtheorem{theorem}{Theorem}[section]

\newtheorem{lemma}[theorem]{Lemma}

\newcommand{\TheTitle}{On Control of Epidemics with Application to
	COVID-19}

\title{{\TheTitle}
}

\author{
	Chung-Han Hsieh\thanks{
		E--mail: \mbox{\href{mailto:chunghan.hsieh@wisc.edu}{chunghan.hsieh@wisc.edu}}.}
}
\date{}
\begin{document}

\maketitle

\begin{abstract}
	At the time of writing, the ongoing COVID-19 pandemic, caused by severe acute respiratory syndrome coronavirus 2 (SARS-CoV-2), had already resulted in more than thirty-two million cases infected and more than one million deaths worldwide.
	Given the fact that the pandemic is still threatening health and safety, it is in the urgency to understand the \mbox{COVID-19} contagion process and know how it might be controlled. 
	With this motivation in mind, in this paper, we consider a version of stochastic discrete-time  Susceptible-Infected-Recovered-Death~(SIRD)-based epidemiological model with two uncertainties: The uncertain rate of infected cases which are undetected or asymptomatic, and the uncertain effectiveness rate of control. 
	Our aim is to study the effect of an epidemic control  policy on the uncertain model in a control-theoretic framework.  
	We begin by providing the closed-form solutions of states in the modified SIRD-based model such as  \textit{infected cases}, \textit{susceptible cases}, \textit{recovered cases} and \textit{deceased cases}. 
	Then, the corresponding expected states and the technical lower and upper bounds for those states are provided as well. 
	Subsequently, we consider two epidemic control problems to be addressed: One is \textit{almost sure epidemic control problem} and the other \textit{average epidemic control problem}. 
	Having defined the two problems, our main results are a set of sufficient conditions on a class of linear control policy which assure that the epidemic is ``well-controlled"; i.e., both of the infected cases and deceased cases are upper bounded uniformly and the number of infected cases converge to zero asymptotically.
	Our numerical studies, using the historical COVID-19 contagion data in the United States,  suggest that our appealingly simple model and control framework can provide a reasonable epidemic control performance compared to the ongoing pandemic~situation.
\end{abstract}



\section{Introduction} \label{Section: Introduction}
At the time of writing, according to the World Health Organization, the ongoing COVID-19 pandemic, caused by severe acute respiratory syndrome coronavirus 2 (SARS-CoV-2), had already resulted in more than thirty-two million cases infected and more than one million deaths worldwide; see~\cite{WHO}, and even worse, the pandemic seems to be ``showing no clear signs of slowing down" in many countries.
To this end, governments are persistently striving to control and  slow down the spread of \mbox{COVID-19}; e.g., minimizing the contact rate by consulting non-pharmaceutical intervention mechanisms such as lock-downs and social distancing or relying on some pharmaceutical interventions such as enhancing medical treatments, and developing possible cure or remedies. 
With this motivation in mind, the focal point for this paper, as a preliminary work,  is to understand how the disease spreads and how it would be possibly~controlled.

In the existing literature, many mathematical epidemiological models are proposed for modeling the virus transmission and evolution of the epidemic.
For example,  the seminal continuous-time SIR compartmental model,  introduced in \cite{kermack1927contribution}, divides 
the individuals into three disjoint classes: the susceptible S class, the infectious I class, and the recovered R class. 
After that, various extensions and ramifications to this model are proposed which include considerations
such as death due to the epidemic, the delay time for infected individuals to become
infectious, quarantining individuals, and spatial effects; e.g., see~\cite{anderson1992infectious, hethcote2000mathematics, bailey1975mathematical, tornatore2005stability}.
In contrast to the continuous-time model, the discrete-time counterpart for the SIR model and its modifications such as SI and SIS model are analyzed in~\cite{allen1994some}.
A stochastic version of the SIR model and its stability is discussed in~\cite{tornatore2005stability}. 
In addition, a comprehensive review along this line of research can be found in~\cite{hethcote2000mathematics}. 
To close this brief literature survey,  we mention  \cite{stewart2020control, burke2020data, alamo2020data,kohler2020robust, calafiore2020modified} which rely on  control-theoretic approaches to mitigate the epidemics.

Consistent with the existing literature, in this paper, we consider a version of the stochastic discrete-time  Susceptible-Infected-Recovered-Death~(SIRD)-based epidemiological model with uncertainties. 
We note that two uncertain quantities are included in our modified SIRD model; i.e., 
the \textit{uncertain rate  of   infected cases which are undetected or asymptomatic} and the \textit{uncertain effectiveness rate of control}.
The necessity of the uncertain rate  of  infected cases can be argued by the following observations:
While the number of infected cases at a given time is crucial, in practice, the \textit{actual} number
of infected cases often remain unknown --- only the number of infected individuals that have been
detected as ``positive" is available.  
Said another way, many infected individuals are indeed undetected;~see \cite{li2020substantial, chen2020time}.
To this end, a typical assumption used in the literature using SIR-based model is that the infected cases are the \textit{actual} infected ones. 
However, this assumption is arguably unrealistic and may lead to erroneous conclusions; e.g., see~\cite{calafiore2020modified,burke2020data}, which justifies the needs for introducing the uncertain rate  of   infected cases into our modified SIRD~model.
To see why the second uncertain quantities ``effectiveness rate of control" is also needed, we argue as follows:
Suppose a government already announced a social distancing or lock-down policy, there might exist some people who are not disciplined or may be reluctant to follow the medical advise for wearing the masks or social distancing. 
In addition, it might be the case that the cure being used is not very effective or only work for a certain group of people. 
These considerations suggest the needs for introducing the uncertain effectiveness on control policy into our model.
It is also worth mentioning that such a consideration, while is important, seems not appear to have been heavily considered in the existing~literature. 

Unlike many COVID-19 related papers aimed at predicting the behavior of the pandemic disease, in this preliminary work, we formulate two epidemics ``control" problems.  That is, we aim at studying the effect of an ``abstract" epidemic control policy on the model in a control-theoretic framework. The word ``abstract" used above is to refer to the fact that we do not characterize specifically how to develop the medical treatments or when to lock-down using our control policy. Instead, we show the cause and effect when the policy is applied and how much it may affect the disease.
To this end, we provide useful analytic expressions for all of the \textit{states}; i.e., textit{infected cases}, \textit{susceptible cases}, \textit{recovered cases} and \textit{deceased cases}, in our modified SIRD-based model.
Furthermore, we show a set of sufficient conditions on a class of linear control policy which assure that the epidemic is ``well-controlled" in the following sense: Both of the infected cases and deceased cases are upper bounded uniformly, and the ratio of infected cases must go down to zero asymptotically.
Our numerical studies, using the historical~COVID-19 contagion data in the United States,  suggest that our appealingly simple model and control framework can provide a reasonable performance compared to the ongoing pandemic~situation.

\subsection{Organization of the Paper}
The remainder of the paper is organized as follows: In Section~\ref{Section: Preliminaries and Problem Formulation} we describe the preliminaries involving the details of our modified SIRD-based epidemiological model. 
Then two problem formulations for epidemic control: \textit{Almost sure epidemic control problem} and \textit{average epidemic control problem}  are stated. 
Subsequently, in
Sections~\ref{Section: Analysis of Dynamics},  detailed analyses on our epidemiological model are provided.
In Sections~\ref{Section:Control of Epidemics in Almost Sure Sense} and \ref{Section: Control of Epidemics in Expected Value}, using a \textit{linear} epidemic control  policy, we provide sufficient conditions on feedback gain so that the epidemic can be properly controlled in either expected value sense or almost sure sense. 
Next, in Section~\ref{Section: Examples Using Historical Data: Taiwan and United State}, we describe a simple model parameter estimation approach and illustrate its use via numerical examples involving historical COVID-19 data
from the United States, the country among those with a relatively high confirmed cases
in the middle of~2020. 
The control performance is also discussed.
Finally, in Section~\ref{Section: Concluding Remarks}, we provide some concluding remarks and promising directions for future research.

\section{Preliminaries and Problem Formulation} \label{Section: Preliminaries and Problem Formulation}
In the sequel, we take $k$ to be the index indicating the day number and $u(k) \geq 0$ to be the corresponding abstract \textit{epidemic control  policy} implemented by governments.\footnote{Again, as mentioned in Section~\ref{Section: Introduction}, the word ``abstract" is to indicate that we do not characterize specifically how to develop the medical treatments or when to lock-down using our control policy. Instead, we show the cause and effect when the policy is applied and how much it may affect the disease. }
We view that the large values of~$u(k)$ correspond to enhanced medical treatments or  stringent non-pharmaceutical interventions such as mandating social distancing with wearing the masks. The smaller values of $u(k)$ might correspond to diminished medical treatments or relaxation of the rules. In the analysis to follow, we assume that the epidemic control  policy $u(k)$ is \textit{causal}; i.e., it only depends on the past and current information that is available at hand---not the future.

\subsection{Epidemiology Model with Uncertainties} For $k=0,1,\dots,$ let  $S(k)$ be the number of \textit{susceptible cases} at the $k$th day,  $I(k)$ be the \textit{infected cases}  at the~$k$th day and $R(k)$ be the \textit{recovered cases} at the~$k$th day, and~$D(k)$ be the \textit{deceased cases} at the~$k$th day.	
Take $N(k)$ be the underlying human population at the $k$th day satisfying $N(k) = S(k)+I(k)+R(k)$.
As mentioned in Section~\ref{Section: Introduction}, many infected individuals may go undetected or asymptomatic~\cite{li2020substantial} and the effectiveness of a control policy $u(k)$ may be uncertain at the time, we introduce two uncertainty quantities: The uncertain \textit{rate  of   infected cases which are undetected or asymptomatic},  call it~$\delta(k)$, and the uncertain \textit{effectiveness rate of control}, call it~$v(k)$, in our model to follow.
Now, with initial values  $S(0):=N(0):=S_0$, $I(0):=I_0>0$,  and  $R(0)=D(0)=0$, we consider a discrete-time stochastic epidemiological model with uncertainties described~by
\begin{align}\label{eq: infected cases dynamics}
	\begin{cases}
		S(k+1) = S(k)-\delta(k)I(k);\\
		I(k+1) =  (1+\delta(k)) I(k)   -  v(k)u(k)- d_I(k)I(k);   \\
		R(k+1)  = R(k) + v(k)u(k); \\
		D(k+1) = D(k)  + d_I(k) I(k)\\
	\end{cases}
\end{align}
where $d_I(k)$ is the \textit{death rate} for infected cases satisfying $d_I(k)  \in [0,d_{\max}]$ with \mbox{$d_{\max}<1$} for all $k$. 
We assume that the sequence $d_I(k)$ are independent and identically distributed~(i.i.d.) random variables.
In addition, we assume that both of $\delta(k)$ and $v(k)$ are i.i.d. random variables with arbitrary distribution but known bounds $0\leq \delta(k) \leq \delta_{\max}$ and $0<v_{\min} \leq v(k) \leq v_{\max} \leq 1$, respectively.\footnote{Technically speaking, we require that $\delta_{\max}$ is in the support of $\delta(k)$  and  $d_{\max}$ is in the support of~$d_I(k)$. Lastly, the bounds $v_{\min}$ and $ v_{\max}$ are in the support for $v(k)$.}
If $v(k) = 1$, it corresponds to the case where the control policy is extremely effective. On the other hand, if $v(k) \approx 0$, it corresponds to an extremely ineffective case; i.e., the people are not disciplined and may be unwilling to follow the medical advise for masks or social distancing; see also a discussion in Section~\ref{Subsection: What if No or Extremely Ineffective Control}.  
In the sequel, we assume further that $I_0\delta_{\max} < S_0$; and $v(k)$, $\delta(k)$, and $d_I(k)$ are mutually independent.
For the sake of notational convenience, we also take the shorthand notations $\overline{\delta}:=\mathbb{E}[\delta(k)]$, $\overline{d_I}:=\mathbb{E}[d_I(k)]$ and $\overline{v}:=\mathbb{E}[v(k)]$ for all~$k$.

\vspace{3mm}
\textbf{Remark on the Epidemiology Model.} As a preliminary work,  while we allow uncertainties on both infected cases $I(k)$ and control $u(k)$, our epidemiological model~\ref{eq: infected cases dynamics} assumes that $u(\cdot)$ can interact with the infected cases instantaneously without any time delay. The action with delay is left for future work; see also Section~\ref{Section: Concluding Remarks}. 
It is also worth mentioning that, if one considers a linear control policy $u(k)=K I(k)$ with constant $K\geq 0$,  then a \textit{basic reproduction ratio} of sorts, call it $R_0(\cdot)$, can be obtained as follows: For $k\geq 0,$
$$
R_0(k) := \frac{\beta(k)}{\gamma(k)+d_I(k)}
$$
where $\beta(k):={\delta(k)(S(k)+I(k)))}/{S(k)}$ is the infectious rate and $\gamma(k):=v(k)K$ is the recovery rate. 
If~$R_0(\cdot)>1$, the number of infected cases is expected to increase; if~$R_0(\cdot)<1$, then the infected cases decrease.
Finally, we should mention that, in terms of the terminology in systems and control theory, the epidemiological dynamics described by equation~(\ref{eq: infected cases dynamics}) is indeed an \textit{uncertain linear time-varying} system. This observation is useful in the following sections to follow.

\subsection{An Almost Sure Epidemic Control  Problem}
In this subsection, we introduce the first epidemic control  problem, which we call the \textit{almost sure epidemic control  problem}.
That is, given the modified SIRD model~\ref{eq: infected cases dynamics}, we seek a control policy~$u(\cdot)$ which assures the following conditions hold:

$(i)$ The ratio of infected cases converge to zero with probability one; i.e., $$\lim_{k \to \infty} \frac{I(k)}{I_0}  = 0$$ with probability one.

$(ii)$ There exists a constant $M_I>0$ such that infected cases $I(k)$ is upper bounded uniformly by $M_I$; i.e., $I(k)\leq M_I$ for all $k$ with probability one.

$(iii)$ There exists a constant $M_D>0$ such that deceased cases $D(k)$ is upper bounded uniformly by $M_D$; i.e., $D(k)\leq M_D$ for all $k$ with probability one.	

\vspace{3mm}
\textbf{Remarks.}
The results related to this almost sure epidemic control  problem is discussed in Section~\ref{Section:Control of Epidemics in Almost Sure Sense} when a linear control policy of the form $u(k)=K I(k)$ with pure gain $K \geq 0$ for all $k\geq 0$ is applied.

While almost sure epidemic control policy, if exists, can be a good candidate to mitigate the pandemic. However, in practice, some potential issues remain. For example, in some cases, the cost for implementing an almost sure epidemic control  policy may be still too expensive in practice; see~Remark \ref{Remarks on Practical Considerations} in Section~\ref{Section:Control of Epidemics in Almost Sure Sense}.
To hedge this issue, we now introduce our second epidemic control  problem which we call the \textit{average epidemic control  problem}  aimed at controlling the ``expected" infected cases and ``expected" deceased cases.

\subsection{An Average Epidemic Control  Problem}
To address the issue raised in previous subsection, we consider the second epidemic control  problem which we call the \textit{average epidemic control  problem}  aimed at controlling the ``expected" infected cases and ``expected" deceased cases. Rigorously speaking,
given the modified SIRD model~\ref{eq: infected cases dynamics}, we seek a control policy $u(\cdot)$ which assures the following conditions hold:

$(i)$ The ratio of expected infected cases $\mathbb{E}[I(k)]/I_0$ converge to zero asymptotically.

$(ii)$ There exists a constant $C_I>0$ such that expected infected cases $\mathbb{E}[I(k)]$ is  upper bounded uniformly by $C_I$; i.e., $\mathbb{E}[I(k)]\leq C_I$ for all $k$.

$(iii)$ There exists a constant $C_D>0$ such that deceased cases $\mathbb{E}[D(k)]$ is  upper bounded uniformly by $C_D$; i.e., $\mathbb{E}[D(k)]\leq C_D$ for all $k$.

\vspace{3mm}
\textbf{Remark.} Similar to the previous remark,
the results related to this average epidemic control  problem is discussed in Section~\ref{Section: Control of Epidemics in Expected Value} when a linear control policy of the form $u(k)=K I(k)$ with $K \geq 0$  for all $k\geq 0$ is applied.
%

\section{Analysis of Epidemiological Model} \label{Section: Analysis of Dynamics}
In this section, to understand the contagion process and the evolution of the epidemic,
we assume that there exists a control policy $u(\cdot)$ which assures that infected cases~$I(k)$, susceptible cases~$S(k)$, recovered cases $R(k)$ and deceased cases~$D(k)$ are all nonnegative for all $k$ with probability one.\footnote{Such a policy exists, for example, by taking $u(k) = K I(k)$ with $K \in [0, (1-d_{\max})/v_{\max}]$; see Section~\ref{Section:Control of Epidemics in Almost Sure Sense} for detailed discussion on this topic.}
The analytic expression for  infected cases, susceptible cases, recovered cases, and deceased cases are provided. The results obtained in this section are useful for the subsequent sections to follow.

\subsection{Analysis of Infected Cases}
In the sequel, we use a shorthand notation  
\begin{align} \label{eq: Phi_I formula}
	\Phi(k,k_0):= \prod_{i=k_0}^{k-1} (1+\delta(i) - d_I(i))
\end{align}
for integer $k\geq k_0\geq 0$.
We mention here that $\Phi(k,k_0)$ is a \textit{state transition function} for infected cases from day $k_0$ to $k$. 
The reader is referred to \cite{chen1999linear} for detailed discussion on this topic. We may sometimes write $\Phi_I(\cdot,\cdot)$ instead of $\Phi(\cdot,\cdot)$ to emphasize that the function is related  to the infected cases. 
The following lemma characterizes some useful properties of the function.

\begin{lemma}[Properties of State Transition Function] \label{lemma: upper bound on Phi_c} 
	For integer $k\geq k_0\geq 0$,	the state transition function $\Phi_I(k,k_0)$ satisfies
	
	(i) $(1-d_{\max})^{k-k_0} \leq \Phi_I(k,k_0) \leq (1+\delta_{\max} )^{k-k_0}$ with probability one.
	
	(ii)
	$
	\mathbb{E}[\Phi_I(k,k_0)] = 	(1+\overline\delta-\overline{d})^{k-k_0}
	$.
\end{lemma}

\begin{proof} To prove part~$(i)$, we let $f(k):=1+\delta(k)-d_I(k)$. Since $\delta(k) \in [0, \delta_{\max}]$ and $d_I(k) \geq 0$ for all $k \geq k_0 \geq 0$, we have 
	$
	1-d_{\max} \leq f(k) \leq  1+\delta_{\max}
	$
	for all $k$ with probability one. Note that the upper and lower bounds are nonnegative. Hence,  a straightforward multiplication leads to the desired inequality and the proof of part~$(i)$ is complete.
	
	To prove part~$(ii)$, with the aid of the fact that $\delta(k)$ and $d_I(k)$ are i.i.d. in $k$ and are mutually independent, we obtain
	\begin{align*}
		\mathbb{E}[\Phi_I(k,k_0)] 
		&= 	\mathbb{E}\left[ \prod_{i=k_0}^{k-1} (1+\delta(i)-d_I(i)) \right]\\
		&= 	 \prod_{i=k_0}^{k-1} \mathbb{E}\left[1+\delta(i)-d_I(i) \right]\\
		& =	 (1+\overline\delta-\overline{d_I})^{k-k_0}
	\end{align*}
	where $\overline\delta=\mathbb{E}[\delta(k_0)]$ and $\overline{d_I}=\mathbb{E}[d_I(k_0)]$, which is desired.
\end{proof}

Having proved Lemma~\ref{lemma: upper bound on Phi_c}, we are now ready to characterize the infected cases~$I(k)$.

\begin{lemma}[Number of Infected Cases] \label{lemma: solution of I_k} Given $I_0>0$,
	the infected cases at the $k$th day is described~by
	\[
	I(k) =  \Phi_I(k,0)I_0 - \sum_{i = 0}^{k-1} \Phi_I(k,i+1)v(i) u(i)
	\]
	for $k \geq 1.$
	Moreover, the expected infected cases satisfy
	\[
	\mathbb{E}\left[ I(k) \right] \leq (1+\overline\delta-\overline{d_I})^{k} I_0 - v_{\min} \sum_{i = 0}^{k-1} (1+ \overline\delta - \overline{d_I})^{k-i-1}\mathbb{E}[u(i)]\\.
	\]
\end{lemma}

\begin{proof} 
	Since the infected cases dynamics~(\ref{eq: infected cases dynamics}) is a linear time-varying system, the solution $I(k)$ and its corresponding proof are well-established;  hence we omitted here and refer the reader to any standard textbook in system theory; e.g., see \cite{chen1999linear}.
	
	To complete the proof, we must show that the expected infected cases can be characterized by the form described in the lemma. Using Lemma~\ref{lemma: upper bound on Phi_c} and the fact that $v(k) \geq v_{\min}$ and are i.i.d., we observe that
	\begin{align*}
		\mathbb{E}\left[ I(k) \right] 
		&= \mathbb{E}[\Phi_I(k,0)]I_0 - \sum_{i = 0}^{k-1} \mathbb{E}[\Phi_I(k,i+1) v(i)u(i)]\\
		&\leq   (1+\overline\delta-d_I)^{k} I_0 - v_{\min} \sum_{i = 0}^{k-1} \mathbb{E}[\Phi_I(k,i+1) u(i)].
	\end{align*}
	Note that $u(i)$, as assumed in Section~\ref{Section: Preliminaries and Problem Formulation}, is causal; i.e., it only depends on the information available up to the $i$th day. In addition, according to equation~(\ref{eq: Phi_I formula}), the state transition function $		\Phi_I(k,i+1)= \prod_{j=i+1}^{k-1} (1+\delta(j) - d_I(j))$, which is independent of~$u(i)$ for $i=0,\dots,k-1$. Therefore, it follows that
	\begin{align*}
		\mathbb{E}\left[ I(k) \right] 
		&= \mathbb{E}[\Phi_I(k,0)]I_0 - \sum_{i = 0}^{k-1} \mathbb{E}[\Phi_I(k,i+1) v(i)u(i)]\\
		&\leq   (1+\overline\delta-d_I)^{k} I_0 - v_{\min} \sum_{i = 0}^{k-1} (1+\overline\delta-\overline{d_I})^{k-i-1}\mathbb{E}[u(i)]. \qedhere
	\end{align*}
\end{proof}

\subsubsection{Remark} As seen later in the next section, if one adopts the linear feedback policy of the form $u(k)=KI(k)$, then the solution $I(k)$ described in Lemma~\ref{lemma: solution of I_k} can be greatly simplified.

\subsection{Analysis of Susceptible Cases}
The following lemma gives a closed-form expression for the susceptible cases~$S(k)$.

\begin{lemma}[Number of Susceptible Cases] \label{lemma: number of susceptible cases}
	With $S(0)=S_0>0$, the number of susceptible cases at the $k$th day is given~by
	\begin{align*}\label{eq: s_k formula}
		S(k)  =  S_0 -\delta(0) I_0- \sum_{i = 1}^{k-1} \delta(i) \left( \Phi_I(i,0)I_0 - \sum_{j = 0}^{i-1} \Phi_I(i,j+1)v(j) u(j) \right)
	\end{align*}
	for $k\geq 1$ where $\Phi_I(k,k_0)$ is the state transition function defined in equation~(\ref{eq: Phi_I formula}) for $k\geq k_0\geq 0$.
\end{lemma}

\begin{proof} Recalling that $S(k+1)=S(k)-\delta(k)I(k)$, a straightforward inductive calculation leads to
	\begin{align} 
		S(k) 
		&=   S_0 - \sum_{i = 0}^{k-1} \delta(i) I(i)\\
		&= S_0 -\delta(0) I_0- \sum_{i = 1}^{k-1} \delta(i) I(i)
	\end{align}
	Using Lemma~\ref{lemma: solution of I_k}, we obtain
	\[
	S(k) = S_0 -\delta(0) I_0- \sum_{i = 1}^{k-1} \delta(i) \left( \Phi_I(i,0)I_0 - \sum_{j = 0}^{i-1} \Phi_I(i,j+1)v(j) u(j) \right). \qedhere
	\]
	%
\end{proof}

\subsection{Analysis of Recovered Cases}
In this subsection, we characterize the solution of recovered cases and some of its useful properties.

\begin{lemma}[Number of Recovered Cases]\label{lemma: Positivity of recovered Cases} With $r(0)=0$,
	the recovered cases at the $k$th day is given by
	$
	R(k) =  \sum_{i = 0}^{k-1}  v(i) u(i)
	$
	for $k \geq 1.$	
	Moreover, the expected number of recovered cases is given by
	$
	\mathbb{E}[R(k)] = 	 \sum_{i = 0}^{k-1}  \mathbb{E}[v(i) u(i)].
	$
	Moreover, it is bounded; i.e.,
	$$
	v_{\min} \sum_{i = 0}^{k-1} \mathbb{E}[ u(i)] \leq 	\mathbb{E}[R(k)] \leq 	 v_{\max}  \sum_{i = 0}^{k-1}  \mathbb{E}[ u(i)].
	$$
	for $k\geq 1.$ 
	
\end{lemma}

\begin{proof} While the proof is simple, for the sake of completion, we provide a full proof here.
	We first recall that $r(0)=0$ by assumption on initial condition. 
	To complete the proof, we proceed a proof by induction.
	Begin by noting that for $k=1$, the recursion on $R(k)$ tells us that
	$
	r(1)  = r(0) + v(0)u(0) = v(0)u(0).
	$
	We now proceed a proof by induction. Assuming that $r(j) =   \sum_{i=0}^{j-1} v(i)u(i)$ for all $j \leq k$, we must show $R(k+1) = \sum_{i=0}^{k}v(i)u(i)$.
	Note that
	\begin{align*}
		R(k+1) 
		&= R(k) + v(k)u(k)\\
		& =  \sum_{i=0}^{k} v(i)u(i)
	\end{align*}
	which is desired.

	
	To complete the proof, we now show the desired expected recovered cases and its bounds. For $k\geq 1$, since $R(k) = \sum_{i = 0}^{k-1} u(i)v(i)$, with the linearity of expected value, it immediately follows that
	$
	\mathbb{E}[R(k)] = 	 \sum_{i = 0}^{k-1}  \mathbb{E}[v(i) u(i)].
	$
	Hence, using the fact that $0<v_{\min} \leq v(k)\leq v_{\max} \leq 1$ for all $k$, it implies that
	\begin{align*}
		v_{\min} \sum_{i = 0}^{k-1} \mathbb{E}[ u(i)] \leq 	\mathbb{E}[R(k)] \leq 	  v_{\max}  \sum_{i = 0}^{k-1}  \mathbb{E}[ u(i)] 
	\end{align*}
which is desired. \qedhere
\end{proof}

\subsubsection{Remark} It is readily seen that $R(k)$ is increasing in $k$ since $u,v$ are nonnegative. 

\subsection{Analysis of Deceased Cases}
The number of deceased cases and its expectation are provided in this subsection.

\begin{lemma}[Number of Deceased Cases] \label{lemma: number of deaths}
	With $D(0)=0$, the number of deceased cases at the $k$th day is given~by
	\begin{align}\label{eq: d_k formula}
		D(k)  =  \sum_{i = 0}^{k-1} d(i) 
		I(i) 
	\end{align}
	for $k\geq 1.$
	Moreover, 
	the expected number of death is bounded by $$0\leq \mathbb{E}[D(k)]  \leq d_{\max}\sum_{i=0}^{k-1}\mathbb{E}[I(i)]  .$$ 
\end{lemma}

\begin{proof} With $D(0)=0$, we begin by recalling that the deaths dynamics~(\ref{eq: infected cases dynamics}) is given by $D(k+1) = D(k) + d_I(k)I(k)$ for $k\geq 0$.
	An almost identical proof as seen in Lemma~\ref{lemma: solution of I_k} leads to 
	$
	D(k)  = \sum_{i = 0}^{k-1} 
	d_I(i) I(i) 
	$
	for $k \geq 1.$ 
	
	To complete the proof, we take expectation on equation~\ref{eq: d_k formula} and, with the linearity of expectation and the fact $0\leq d_I(k)\leq d_{\max}$, the desired inequality follows immediately.
\end{proof}


\subsection{What if No or Extremely Ineffective Control}\label{Subsection: What if No or Extremely Ineffective Control}
If there are low or no medical supports, i.e., $u(k) \approx 0$, or the medical treatments are lacking or extremely ineffective, due to the reasons such as the people are not disciplined or unwilling to follow the medical advise to wearing masks or social distancing; i.e., $v(k) \approx 0$, then it is readily verified that the susceptible cases, infected cases, recovered cases, and deceased cases become
\begin{align}\label{eq: c_r_d with no control}
	\begin{cases}
		S(k) \approx  S_0 -\left( \delta(0) + \sum_{i = 1}^{k-1} \delta(i)  \left( \prod_{j=0}^{i-1} (1+\delta(j) - d_I(j)) \right) \right) I_0  \\
		I(k) \approx  \prod_{i=0}^{k-1} (1+\delta(i) - d_I(i) )I_0; \\
		R(k) \approx  0;\\
		D(k) \approx    \sum_{i = 0}^{k-1}d_I(i)  \left(  \prod_{j=0}^{i-1} (1+\delta(j) - d_I(j))  I_0 \right). 
	\end{cases}
\end{align}

Consistent with the intuition, the equation~(\ref{eq: c_r_d with no control}) above tells us that the recovered cases $R(k)$ are close to zero approximately. On the other hand, the infected cases $I(k)$ and the deceased cases~$D(k)$ are both increasing monotonically; i.e., an ``outbreak" is seen.

\subsection{Linear Epidemic Control  Policy} 
To mitigate the pandemic, in this paper, we consider a class of \textit{linear epidemic control  policy} of the form 
$
u(k) :=K I(k)
$
for $k\geq 0$ 
where $K \geq 0$ is the \textit{feedback gain} which represents the degree of the epidemic control  effort on providing medical treatment or mandating the non-pharmaceutical interventions.
In the sequel, we sometimes call the linear epidemic control  policy above as \textit{linear feedback policy} or just \textit{linear policy} for the sake of brevity.

\section{Control of Epidemics in Almost Sure Sense} \label{Section:Control of Epidemics in Almost Sure Sense}
As mentioned in Section~\ref{Section: Preliminaries and Problem Formulation}, our objective of this section is to address the \textit{almost sure epidemic control  problem} using  the linear control policy. We begin with discussing the control of infected cases in the  almost sure sense.

\subsection{Control of Infected Cases}
The following lemma is useful for deriving one of our main results: Theorem~\ref{theorem: Linear Controlled Infected Cases}.

\begin{lemma}[Boundedness of Infected Cases]\label{lemma: I_k solution under linear policy}
	For any linear  control policy of the form $u(k)=K I(k)$ with gain $0 \leq K \leq {(1-d_{\max})}/{v_{\max}}$, we have
	
	(i)
	$
	0\leq I(k) \leq  (1+\delta_{\max}-K v_{\min})^k I_0$
	for all $k\geq 1$ with probability one.
	
	(ii)
	If $ K < (1-d_{\max})/{v_{\max}}$, then $I(k)>0$ for all $k\geq 1$ with probability one. 
\end{lemma}

\begin{proof}
	Fix $k\geq 1$ and let $u(k)=KI(k)$ with $K \in \left[0, \frac{1-d_{\max}}{v_{\max}} \right]$. 
	Then the infected cases dynamics becomes $I(k+1) = (1 + \delta(k) - d_I(k) - Kv(k))I(k)$, which implies that
	\[
	I(k) = \prod_{i=0}^{k-1}f(i)I_0
	\]
	where
	$
	f(i):= 1+\delta(i) -d_I(i) - K v(i) .
	$
	Since
	$$ 
	f(i) \geq  1 -d_{\max} - K v_{\max}
	$$
	and $K  \leq {(1-d_{\max})}/{v_{\max}}$, we conclude that the $f(i)\geq 0$ for all $i=0,\dots, k-1$. 
	Therefore, the product of $f(i)$ is also nonnegative which implies that $I(k) \geq 0$ for all~$k$ with probability one. 
	To complete the proof of part~$(i)$, we note that for any $K \in [0, \frac{1-d_{\max}}{v_{\max}}]$, the infected cases satisfy
	\begin{align*}
		I(k) &= \prod_{i=0}^{k-1}(1+\delta(i) - d_I(i)-K v(i))I_0 \\
		& \leq (1+\delta_{\max}-K v_{\min})^k I_0
	\end{align*}
	which is desired.
	
	To  prove part~$(ii)$, an almost identical proof as seen in part~$(i)$ applies. That is, 
	if $K < {(1-d_{\max})}/{v_{\max}}$, then $I(k)>0$ for all $k$ with probability one.  
\end{proof}

\subsubsection{Remark} According to part~$(ii)$ of the lemma above, we see that if $K<(1-d_{\max})/v_{\max}$, then the infected cases $I(k)$ is strictly positive for all $k$ with probability one. 
Said another way, the infected cases can not be eradicated at any time. 
However, as seen in the next theorem, the ratio of the infected cases can go down to zero asymptotically for some $K$. 

\begin{theorem}[Asymptotic Behavior and Boundedness on Infected Cases]\label{theorem: Linear Controlled Infected Cases} If the maximum uncertain rate of infected cases  satisfies 
	$
	\delta_{\max} < \frac{v_{\min}(1-d_{\max})}{v_{\max}}
	$ and feedback gain
	\begin{align} \label{eq: K_k condition}
		\frac{\delta_{\max}}{v_{\min}} < K  < \frac{1-d_{\max}}{v_{\max}},
	\end{align} 
	then we have
	
	(i)
	$
	\lim_{k\to \infty} I(k)/I_0 = 0
	$ with probability one. 
	
	(ii) The controlled infected cases are upper bounded uniformly; i.e.,
	$
	I(k)  < I_0
	$
	for all $k$ with probability one. If $K=\frac{\delta_{\max}}{v_{\min}} $, then $I(k) \leq I_0.$ 
	
\end{theorem}

\begin{proof} 
	To prove $(i)$, we observe that $I(k) = \prod_{i=0}^{k-1}f(i)I_0$ for $k\geq 1$ where
	$$
	f(i)= 1+\delta(i) - d_I(i)  - K v(i) 
	$$
	as defined in the proof of Lemma~\ref{lemma: I_k solution under linear policy}.
	With the assumed inequalities $$
	\delta_{\max} < \frac{v_{\min}(1-d_{\max})}{v_{\max}}
	$$ and feedback gain
	$
	\frac{\delta_{\max}}{v_{\min}} < K  < \frac{1-d_{\max}}{v_{\max}},
	$
	it is readily verified that $0<f(i) < 1$ for all $i =0,1,\dots, k-1$.
	To show the limit of the ratio $I(k)/I_0$ is zero with probability one,
	%
	we write
	\begin{align*}
		\prod_{i=0}^{k-1}f(i) = \exp\left( k \frac{1}{k}  \sum_{i=0}^{k-1} \log f(i) \right).
	\end{align*}
	We also note that the logarithmic function above is well-defined for $K$ within the assumed range.
	Hence, 
	\begin{align*} \label{eq: limit of I_k}
		\lim_{k \to \infty} \frac{I(k)}{I_0} & = \lim_{k \to \infty} \prod_{i=0}^{k-1}f(i)\\
		&= \lim_{k \to \infty} \exp\left(k \frac{1}{k} \sum_{i=0}^{k-1} \log f(i) \right)\\
		&=  \exp\left( \lim_{k \to \infty}k \frac{1}{k} \sum_{i=0}^{k-1} \log f(i) \right).
	\end{align*}
	Now, since $f(i)$ are i.i.d., by the strong law of large number (SLLN); e.g., see~\cite{durrett2019probability}, we have $$\frac{1}{k} \sum_{i=0}^{k-1} \log f(i) \to \mathbb{E}[\log f(0)]$$ as $k \to \infty$ with probability one.
	Now, using the fact that the logarithmic function is strictly concave, Jensen's inequality yields
	\begin{align*}
		\mathbb{E}[\log(f(0))] 
		&< \log \mathbb{E}[f(0)]\\
		&= \log (1+\overline{\delta}-\overline{d_I} -K\overline{v}) \\
		&\leq \log (1+\overline{\delta} -K\overline{v}).
	\end{align*}
	where the last inequality hold by the monotonicity of logarithmic function. Hence,
	$\log (1+\overline{\delta} -K\overline{v}) \leq 0$
	for any $K \geq  \delta_{\max}/v_{\min} \geq \overline{\delta}/\overline{v}$.   
	Therefore, $\mathbb{E}[\log f(0)] < 0$, which
	implies that $\lim_{k \to \infty} I(k)/I_0 \to 0$ and the proof for part~$(i)$ is complete.
	
	To prove $(ii)$,
	we must show that the $I(k)$ is upper bounded by the initial infected cases $I_0$. 
	As seen in the proof of part~$(i)$ of Lemma~\ref{lemma: I_k solution under linear policy}, we have, for all $k$,
	\begin{align}  
		I(k) \leq \prod_{i=0}^{k-1}(1+\delta_{\max}-K v_{\min})I_0 \leq I_0
	\end{align}
	where the last inequality holds by using the assumed fact that $K \geq \frac{\delta_{\max}}{v_{\min}}$. \qedhere
	
	%
\end{proof}

\subsubsection{Practical Considerations and Sub-Optimal Feedback Gain} \label{Remarks on Practical Considerations}
In practice, since the medical resource is limited, it is natural to put additional constraints on the feedback gain such as $K\in[0, L]$ with $L\geq 1$ being the constant which corresponds to the maximum allowable  medical resources to be used. 
Hence, other than the sufficient condition~\ref{eq: K_k condition}, for any $\delta_{\max} < \frac{v_{\min}(1-d_{\max})}{v_{\max}}$, one might require~$K$ to satisfy
\begin{align} \label{eq: optimal K}
	K \in \mathcal{K}:= \left( \frac{\delta_{\max}}{v_{\min}}, \,  \frac{1-d_{\max}}{v_{\max}}\right)  \cap [0,L].
\end{align}
However, one should note that the set $\mathcal{K}$ might be an empty set. 
To see this, we consider $L=1$, $v_{\min}=v_{\max}= 0.01$ and $\delta_{\max}= 0.1$ and $d_{\max} = 0.01$. Then it is readily verified that the condition	$
\delta_{\max} < \frac{v_{\min}(1-d_{\max})}{v_{\max}}
$ but
$
\mathcal{K}  = \left( 10, \,  99\right)\cap [0,1] = \emptyset. 
$
%
%
If $\mathcal{K} = \emptyset$, it tells us that the infected cases does not converge to zero asymptotically and the epidemic is \textit{uncontrollable}. If this is the case, one should put whatever they can to suppress the disease; e.g., by putting $K:=L$ for possible $L\geq 1$.

On the other hand, given the fact that the higher feedback gain would cause a higher consumption on medical/economical resources, one should choose the lowest feedback gain which is nonzero. Hence,
an immediate ``sub-optimal" choice would simply be  as follows: If $\delta_{\max}\leq v_{\min}$  we take
$
K^*:={\delta_{\max}}/{v_{\min}}
$
with the aim that $I(k) \leq I_0$ can be guaranteed at least.
On the other hand, if $\delta_{\max}>v_{\min}$, we take
$
K^*:= L.
$
Or, we can write it in a more compact way as follows:
\begin{align}
	K^* := \frac{\delta_{\max}}{v_{\min}}\cdot 1_{\{ \delta_{\max}\leq L v_{\min}\} } + L\cdot 1_{\{\delta_{\max}> L v_{\min}\} }
\end{align}
provided that $\delta_{\max} < \frac{v_{\min}(1-d_{\max})}{v_{\max}}$ where $1_{\{ \delta_{\max}\leq L v_{\min}\} }$ and $1_{\{\delta_{\max}> L v_{\min}\} }$ are the indicator function; see  \cite{gubner2006probability} for a detailed discussion on this topic. 
As seen later in Section~\ref{Section: Examples Using Historical Data: Taiwan and United State}, the $K^*$ above will be adopted to test the epidemic control  performance using historical data.


\subsection{Recovered Cases, Susceptible Cases, and Deceased Cases}
In this subsection, the recovered cases $R(k)$, susceptible cases $S(k)$, and deceased cases $D(k)$ under linear control policy are discussed.
The analysis of recovered cases is simple. We begin by recalling Lemma~\ref{lemma: Positivity of recovered Cases} and write
$
R(k) = \sum_{i = 0}^{k-1} v(i)u(i).
$
Now, by taking linear control policy $u(k)=KI(k)$ with  gain $0 \leq K \leq {(1-d_{\max})}/{v_{\max}}$, we have $R(k)\geq 0$ for all $k$ with probability one and
\begin{align*}
	R(k) 
	&\leq v_{\max} K  \sum_{i = 0}^{k-1}   (1+\delta_{\max}-K v_{\min})^i I_0
\end{align*}
where the last inequality holds by using Lemma~\ref{lemma: I_k solution under linear policy}.

\begin{lemma}[Susceptible Cases Under Linear Policy]\label{lemma: s_k solution via linear policy}
	Any linear  control policy of the form $u(k)=KI(k)$ with gain $0 \leq K \leq {(1-d_{\max})}/{v_{\max}}$ leads to
	\begin{align*} 
		S(k) 
		= S_0 -\delta(0) I_0- \sum_{i = 1}^{k-1} \delta(i) \left(  \prod_{j=0}^{i-1}(1+\delta(j) - d_I(j)-K v(j))I_0  \right) \geq 0
	\end{align*}
	for $k=1,2,\dots, \floor{S_0/(I_0\delta_{\max})}$	with probability one where $\floor{z}$ is the floor function satisfying $$\floor{z}:=\max\{n\in \mathbb{Z}: n\leq z\}.$$
\end{lemma}

\begin{proof} 
	Similar to the proof of Lemma~\ref{lemma: number of susceptible cases}, recalling that $S(k+1)=S(k)-\delta(k)I(k)$, a straightforward inductive calculation leads to
	\begin{align*} 
		S(k) 
		&=   S_0 - \sum_{i = 0}^{k-1} \delta(i) I(i)\\
		&= S_0 -\delta(0) I_0- \sum_{i = 1}^{k-1} \delta(i) I(i).
	\end{align*}
	Using Lemma~\ref{lemma: I_k solution under linear policy}, the desired form of solution follows immediately. To complete the proof, we note that
	\begin{align*}
		S(k) &= S_0 -\delta(0) I_0- \sum_{i = 1}^{k-1} \delta(i) I(i)\\
		&\geq S_0 -\delta_{\max} I_0- \delta_{\max} \sum_{i = 1}^{k-1} (1+\delta_{\max}-K v_{\min})^i I_0\\
		&=\begin{cases}
			\frac{\delta_{\max} I_0  (\delta_{\max}-K v_{\min}+1)^k-\delta_{\max} (I_0 +S_0)+K S_0 v_{\min}}{K v_{\min}-\delta_{\max}}, & \delta_{\max}<Kv_{\min};\\
			S_0 - I_0  \delta_{\max}k & \delta_{\max} = Kv_{\min}.
		\end{cases}
	\end{align*}
	Thus, via a lengthy but straightforward calculations, it is readily verified that $S(k)\geq 0$ for $k\leq \floor{S_0/(I_0\delta_{\max})}$ with probability~one.
	%
\end{proof}

\textbf{Remark.} While the lemma above tells us that $S(k)\geq 0$ holds up to $k\leq \floor{S_0/(I_0\delta_{\max})}$, we should note that the initial cases~\mbox{$S_0=N(0)$} are often far larger than the denominator $I_0\delta_{\max}$. Hence, without loss of generality, in the sequel, we deemed that $S(k)\geq 0$ for sufficiently large $k$. 
The next lemma indicates that the deceased cases under the linear control policy.

\begin{lemma}[Upper Bound on The Deceased Cases Under Linear Policy] \label{lemma: control of upper bounds on death} 	Any linear policy of the form $u(k)=KI(k)$ with gain $\delta_{\max}/v_{\min} \leq K \leq {(1-d_{\max})}/{v_{\max}}$ leads to the deceased cases satisfying
	\begin{align*}
		D(k)  
		&\leq \begin{cases}
			\frac{d_{\max} I_0  \left(1-(\delta_{\max}-K v_{\min}+1)^k \right)}{K v_{\min}-\delta_{\max}}, & \delta_{\max}  < K v_{\min};\\
			d_{\max} I_0 k, & \delta_{\max}= K v_{\min}.
		\end{cases}
	\end{align*}
	for all $k\geq 1$ with probability one.
\end{lemma}

\begin{proof}
	We begin by recalling that $	D(k)  =  \sum_{i = 0}^{k-1} d_I(i)
	I(i) $ for all $k\geq 1$. Hence, using Lemma~\ref{lemma: I_k solution under linear policy}, we have
	\begin{align*}
		D(k)  &= \sum_{i = 0}^{k-1} d_I(i)
		I(i)\\
		&\leq d_{\max} I_0 \sum_{i = 0}^{k-1} 
		(1+\delta_{\max} -K v_{\min})^i.
	\end{align*}
	Note here that the assumption $K \leq (1-d_{\max})/v_{\max}$ assures $I(k)\geq 0$ for all~$k$ with probability one; hence the inequality above is well-defined.
	To complete the proof, with the aid of sum of geometric series, we have
	\begin{align*}
		D(k)  
		&\leq d_{\max} \sum_{i = 0}^{k-1} 
		(1+\delta_{\max} -K v_{\min})^i I_0\\
		&=\begin{cases}
			\frac{I_0 d_{\max} \left(1-(\delta_{\max}-K v_{\min}+1)^k \right)}{K v_{\min}-\delta_{\max}}, & \delta_{\max}  < K v_{\min};\\
			d_{\max} I_0 k, & \delta_{\max}= K v_{\min}
		\end{cases}
	\end{align*}
	which completes the proof.
\end{proof}

\subsection{Almost Sure Epidemic Control Policy: A Revisit} With the aids of Theorem~\ref{theorem: Linear Controlled Infected Cases} and Lemma~\ref{lemma: control of upper bounds on death}, we see that if we take linear feedback policy with constant gain; i.e., $u(k)=KI(k)$ and assuming that $\delta_{\max}<v_{\min}(1-d_{\max})/v_{\max}$ and $K\in \left( \frac{\delta_{\max}}{v_{\min}}, \,  \frac{1-d_{\max}}{v_{\max}}\right)$, then we have: $(i)$ infected cases converges to zero asymptotically with probability one. $(ii)$ there exists a constant $M_I:=I_0>0$ such that infected cases $I(k)\leq M_I$  for all $k$ with probability one. $(iii)$  there exists a constant $$M_D:=\frac{I_0 d_{\max} }{K v_{\min}-\delta_{\max}}>0$$ such that infected cases $I(k)\leq M_D$ for all $k$ with probability one. Therefore, the \textit{almost sure epidemic control problem} is solved. 

While the almost sure epidemic control policy, if exists, can be a good candidate to mitigate the pandemic,  some potential issues remain in practice. That is,
recalling Remark~\ref{Remarks on Practical Considerations} in Section~\ref{Section:Control of Epidemics in Almost Sure Sense}, it tells us that, in some cases, the almost sure epidemic control  policy, while exist, may not be possible to implement.
To address this, as mentioned in Section~\ref{Section: Preliminaries and Problem Formulation}, we now move to our second epidemic control  problem which we call the \textit{average epidemic control  problem}. In this setting, the  aim now becomes to control the ``expected" infected cases and ``expected" deceased cases

\section{Control of Epidemics in the Sense of Expected Value} \label{Section: Control of Epidemics in Expected Value}
In this section, we provide our results on control of epidemics in the sense of expected value. We begin with discussing the control of expected infected cases.

\subsection{Control of Expected Infected Cases}
The lemma below provides an useful analytical expression for the expected number of infected cases.

\begin{lemma}[Expected Value of Infected Cases]\label{lemma: Expected c_k} 
	The expected number of infected cases is given by
	$
	\mathbb{E}[I(k)] = 	(1+\overline{\delta}- \overline{d_I}-K \overline{v})^k I_0
	$
	for any linear policy $u(k)=KI(k)$ with gain $K \in \left[0, {(1-d_{\max})}/v_{\max} \right]$.
\end{lemma}

\begin{proof} The proof is straightforward. We begin by noting that, for  any linear policy $u(k)=KI(k)$ with gain $K\in \left[0, {(1-d_{\max})}/v_{\max} \right]$, $I(k) \geq 0$ for all $k$ with probability one and
	$
	I(k) = \prod_{i=0}^{k-1}(1+\delta(i) - d_I(i)-K v(i))I_0.
	$
	Since $\delta(k), d_I(k)$ and~$v(k)$ are i.i.d in~$k$ and are mutually independent, taking the expected value on the $I(k)$ above yields
	\begin{align*}
		\mathbb{E}[I(k)] 
		&= \prod_{i=0}^{k-1}	\mathbb{E}[1+\delta(i)-d_I(i)-K v(i)]I_0 \\
		& =  	\left( 1+\overline{\delta}-\overline{d_I}-K \overline{v} \right)^kI_0
	\end{align*}
	and the proof is complete.   
\end{proof}

With the aid of the lemma above, we are now ready to provide our second main result.

\begin{theorem}[Control of Expected Infected Cases]\label{theorem: Linear Controlled on Expected Infected Cases} If $\overline{\delta}< \frac{(1-d_{\max})\overline{v}}{v_{\max}}$ and the feedback gain satisfies
	\begin{align}  \label{eq: K_k condition in expected value sense}
		\frac{\overline{\delta}-\overline{d_I}}{\overline{v}} <  K < \frac{1-d_{\max}}{v_{\max}},
	\end{align} 
	then we have
	
	(i) 
	$
	\lim_{k\to \infty} \mathbb{E}[I(k)]/I_0 =0.
	$
	
	(ii)  
	The expected infected cases are upper bounded uniformly by initial infected cases; i.e.,
	$
	\mathbb{E}[I(k)] \leq I_0
	$
	for all $k\geq 0$.
	
	(iii) If $K:=(\overline{\delta}-\overline{d_I})/\overline{v}$, then we have
	$
	\mathbb{E}[I(k)] = I_0
	$ for all $k$.

\end{theorem}

\begin{proof} 
	The idea of the proof is similar to the one used in Theorem~\ref{theorem: Linear Controlled Infected Cases}.	However, for the sake of completeness, we provide our full proof here. 
	We begin by assuming that
	$\overline{\delta}< \frac{(1-d_{\max})\overline{v}}{v_{\max}}$ and
	$ 
	\frac{\overline{\delta}-\overline{d_I}}{\overline{v}} <  K < \frac{1-d_{\max}}{v_{\max}}.
	$
	Since $K <  (1-d_{\max})/v_{\max}$, it implies that $I(k) > 0$ for all $k$ with probability one.
	Now, according to Lemma~\ref{lemma: Expected c_k}, we have
	\begin{align}\label{eq: E_I_k}
		\mathbb{E}[I(k)] =  (\mathbb{E}[f(0)])^k	I_0
	\end{align}
	where $\mathbb{E}[f(0)] = 1+\overline{\delta}-\overline{d_I}-K \overline{v}.$
	To see the desired limiting result,
	%
	we now write
	$
	(\mathbb{E}[f(0)])^k = \exp\left(  k\log \mathbb{E}[f(0)] \right).
	$
	Hence, 
	\begin{align*}
		\lim_{k \to \infty} \frac{\mathbb{E}[I(k)]}{I_0} 
		&=  \exp\left( \lim_{k \to \infty} k \log \mathbb{E}[f(0)] \right).
	\end{align*}
	Note that 
	\begin{align*}
		\log \mathbb{E}[f(0)] 
		&= \log (1+\overline{\delta}-\overline{d_I} -K\overline{v}).
	\end{align*}
	Hence, $\log (1+\overline{\delta} -\overline{d_I} -K\overline{v}) < 0$
	for any $K > (\overline{\delta}-\overline{d_I})/\overline{v} $ and the logarithmic function is well-defined for $K$ within the assumed range. Therefore, we have $\log \mathbb{E}[f(0)]<0$, which implies that $\lim_{k \to \infty} \mathbb{E}[I(k)]/I_0 \to 0$ and the proof for part~$(i)$ is complete.

	To prove part~$(ii)$, we fix $k$ and simply note  that, with the assumed assumptions on $\overline{\delta}$ and~$K$ and the fact that $d_I(\cdot)\geq 0$, it is readily verified that $\mathbb{E}[I(k)]\leq I_0$, which completes the proof of part~$(ii)$.
	
	To  prove part~$(iii)$, take $K=(\overline{\delta}-\overline{d_I})/\overline{v}$ and substitute it back into the equation~\ref{eq: E_I_k}, we obtain $\mathbb{E}[I(k)]=I_0$ for all $k$
	and  the proof of part~$(iii)$ is complete.
\end{proof}

\subsubsection{Remarks on Practical Considerations} 
Similar to Remark~\ref{Remarks on Practical Considerations}, to accommodate the practical considerations, we fix $L\geq 1$ and require
\begin{align}\label{cond: K_k condition}
	K \in \left( \frac{\overline{\delta}-\overline{d_I}}{\overline{v}}, \frac{1-d_{\max}}{v_{\max}} \right) \cap [0, L].
\end{align}
Due to the limited available medical resources, it is reasonable to choose a ``sub-optimal" $K$ in the sense of minimizing the potential use of medical resources. 
For example, we let~$1_{\{ \overline{\delta} \leq L \overline{v}\}}$ and $1_{\{\overline{\delta} > L \overline{v}\}}$ be the indicator functions and~consider
\begin{align}
	{K}^* := K = \frac{\overline{\delta}}{\overline{v}}\cdot 1_{\{ \overline{\delta} \leq L\overline{v}\} } + L\cdot 1_{\{\overline{\delta} > L \overline{v}\} }
\end{align}
provided that $\overline{\delta} < \frac{\overline{v}(1-d_{\max})}{v_{\max}}$
to be our choice of feedback gain for controlling of epidemics in the sense of expected value. Note that $K^*> \frac{\overline{\delta} - \overline{d_I}}{\overline{v}}$ for $\overline{d_I}>0$; hence, the theorem above applies.  This usage of $K^*$ will be also seen later in Section~\ref{Section: Examples Using Historical Data: Taiwan and United State}.



\subsection{Recovered Cases, Susceptible Cases, and Deceased Cases}
In this subsection, the expected recovered cases $\mathbb{E}[R(k)]$, expected susceptible cases $\mathbb{E}[S(k)]$, and expected deceased cases $\mathbb{E}[D(k)]$ under linear control policy are discussed.

The analysis of recovered cases is simple. We begin by recalling Lemma~\ref{lemma: Positivity of recovered Cases} and write
$
R(k) = \sum_{i = 0}^{k-1} v(i)u(i).
$
Now, by taking linear control policy $u(k)=KI(k)$ with  gain $0 \leq K \leq \frac{1-d_{\max}}{v_{\max}}$ for all $k$, we have
\begin{align*}
	\mathbb{E}[R(k)] &= \sum_{i = 0}^{k-1} \mathbb{E}[v(i)K I(i)] \\
	&\leq v_{\max} K \sum_{i = 0}^{k-1}  \mathbb{E}[I(i)] \\
	&= v_{\max} K \sum_{i = 0}^{k-1}   	(1+\overline{\delta}- \overline{d_I}-K \overline{v})^iI_0
\end{align*}
where the last inequality holds by using Lemma~\ref{lemma: Expected c_k}. The next lemma states the expected susceptible cases.

\begin{lemma}
	For $k= 1,2,\dots, \floor{S_0/(\delta_{\max}I_0)}$, any linear control policy of the form $u(k) = K I(k)$ with $0\leq K \leq (1-d_{\max})/v_{\max}$ yields the expected susceptible cases
	$$
	\mathbb{E}[S(k)] 
	= \begin{cases}
		S_0 - \frac{\overline{\delta}  I_0 \left(1-(1+\overline{\delta} -\overline{d_I}-K \overline{v})^k\right)}{K\overline{v}-\overline{\delta} +\overline{d_I}}, & K\overline{v}> \overline{\delta} -\overline{d_I};\\
		S_0 - I_0  \overline{\delta}k, & K\overline{v}  = \overline{\delta} -\overline{d_I}
	\end{cases}
	$$
\end{lemma}

\begin{proof} Let $u(k) = K I(k)$ with $0\leq K \leq (1-d_{\max})/v_{\max}$. We begin by recalling that Lemma~\ref{lemma: s_k solution via linear policy} tells us that $S(k)\geq 0$ for stage $k=1,2,\dots, \floor{S_0/(\delta_{\max}I_0)}$.  
	Now	using the facts that $\delta(k)$ are i.i.d. in $k$, and~$\delta$, $d_I$ and $v$ are mutually independent, we~have
	\begin{align*}
		\mathbb{E}[S(k)] 
		&=    S_0-\overline{\delta} I_0 - \sum_{i = 1}^{k-1} \mathbb{E}\left[\delta(i) \left(\prod_{j=0}^{i-1}(1+\delta(j) - d_I(j)-K v(j))I_0 \right)\right]\\
		&=    S_0 -\overline{\delta} I_0 - \sum_{i = 1}^{k-1} \mathbb{E}\left[\delta(i)\right] \mathbb{E}\left[ \prod_{j=0}^{i-1}(1+\delta(j) - d_I(j)-K v(j))I_0 \right]\\
		&=S_0 -\overline{\delta}I_0 \left[ 1  +  \sum_{i = 1}^{k-1}  (1+\overline{\delta}  - \overline{d_I}-K \overline{v} )^i \right].
	\end{align*}
	Using the geometric series, we conclude
	$$
	\mathbb{E}[S(k)] 
	= \begin{cases}
		S_0 - \frac{\overline{\delta}  I_0 \left(1-(1+\overline{\delta} -\overline{d_I}-K \overline{v})^k\right)}{K\overline{v}-\overline{\delta} +\overline{d_I}}, & K\overline{v}> \overline{\delta} -\overline{d_I};\\
		S_0 - I_0  \overline{\delta}k, & K\overline{v}  = \overline{\delta} -\overline{d_I},
	\end{cases}
	$$
	which completes the proof.
\end{proof}


\begin{lemma}[Expected Deceased Cases Under Linear Policy] \label{lemma: control of expected death}
	For $k\geq 1$, any linear feedback control policy of the form $u(k) = K I(k)$ with $0\leq K \leq (1-d_{\max})/v_{\max}$ yields the expected deceased cases 
	\begin{align*}
		\mathbb{E}[D(k)] 
		&= \begin{cases}
			\frac{\overline{d_I} I_0 \left(1-(1+\overline{\delta} -\overline{d_I}-K \overline{v})^k\right)}{K \overline{v}-\overline{\delta} +\overline{d_I}} , & K\overline{v}-\overline{\delta}+\overline{d_I} > 0;\\
			\overline{d_I} I_0 k, & K\overline{v}-\overline{\delta}+\overline{d_I} = 0,
		\end{cases}
	\end{align*}
\end{lemma}

\begin{proof}The proof is similar to Lemma~\ref{lemma: control of upper bounds on death}. 
	Using Lemma~\ref{lemma: I_k solution under linear policy} and the fact that~$v(k)$, $d_I(k)$, and $\delta(k)$ are independent, we have
	\begin{align*}
		\mathbb{E}[D(k)] 
		&=     \sum_{i = 0}^{k-1} 
		\mathbb{E}[d_I(i) I(i)]\\
		&=  \mathbb{E}[d_I(0) I(0)]+   I_0\sum_{i = 1}^{k-1} 
		\mathbb{E}\left[d_I(i) \prod_{j=0}^{i-1}(1+\delta(j) - d_I(j)-K v(j))\right]\\
		&=   \overline{d_I}I_0 +   I_0\sum_{i = 1}^{k-1} 
		\mathbb{E}\left[d_I(i) \prod_{j=0}^{i-1}(1+\delta(j) - d_I(j)-K v(j))\right].
	\end{align*}
	It is easy to verify that $d_I(i)$ and $(1+\delta(j) - d_I(j)-K v(j))$ are independent for all $j=0,\dots,i-1$; hence, we have
	\begin{align*}
		\mathbb{E}[D(k)] 
		&=   \overline{d_I}I_0 +   I_0\sum_{i = 1}^{k-1} 
		\mathbb{E}\left[d_I(i) \prod_{j=0}^{i-1}(1+\delta(j) - d_I(j)-K v(j))\right]\\
		&=      \overline{d_I}I_0 +   I_0\sum_{i = 1}^{k-1} 
		\mathbb{E}\left[d_I(i)\right] \mathbb{E}\left[\prod_{j=0}^{i-1}(1+\delta(j) - d_I(j)-K v(j))\right]\\
		&=   \overline{d_I}I_0 \left[1 +   \sum_{i = 1}^{k-1}  \left( 	1+\overline{\delta}-\overline{d_I}-K \overline{v} \right)^i\right].
	\end{align*}
	Using the geometric series, the equality above reduces~to
	\begin{align*}
		\mathbb{E}[D(k)] 
		&=   \overline{d_I}I_0 \left[1 +   \sum_{i = 1}^{k-1}  \left( 	1+\overline{\delta}-\overline{d_I}-K\overline{v} \right)^i\right]\\ 
		&= \begin{cases}
			\frac{\overline{d_I} I_0 \left(1-(1+\overline{\delta} -\overline{d_I}-K \overline{v})^k\right)}{K \overline{v}-\overline{\delta} +\overline{d_I}} , & K\overline{v}-\overline{\delta}+\overline{d_I} > 0;\\
			\overline{d_I} I_0 k, & K\overline{v}-\overline{\delta}+\overline{d_I} = 0,
		\end{cases}
	\end{align*}
	which is desired.
	%
\end{proof}

\subsubsection{Remark}
For $u(k)=KI(k)$, one can readily verify that
\begin{align*}
	\lim_{k\to \infty} \mathbb{E}[D(k)] 
	&= \begin{cases}
		\frac{ I_0 \overline{d_I}}{K\overline{v}-\overline{\delta}+\overline{d_I}}, & K\overline{v}-\overline{\delta}+\overline{d_I} > 0;\\
		\infty, & K\overline{v}-\overline{\delta}+\overline{d_I} = 0.
	\end{cases}
\end{align*}

\subsection{Average Epidemic Control Policy: A Revisit} With the aids of Theorem~\ref{theorem: Linear Controlled on Expected Infected Cases} and Lemma~\ref{lemma: control of expected death}, we see that if we take linear feedback policy with a pure constant gain; i.e., $u(k)=KI(k)$ and assuming that $\overline{\delta}< \frac{(1-d_{\max})\overline{v}}{v_{\max}}$ and $\frac{\overline{\delta} - \overline{d_I}}{\overline{v}} <   K  < \frac{1-\overline{d_I}}{\overline{v}}$ with $\overline{d}_I >0$, then we have: $(i)$ Expected infected cases converges to zero asymptotically. $(ii)$ There exists a constant $C_I:=I_0>0$ such that expected infected cases $\mathbb{E}[I(k)]\leq C_I$  for all~$k$. $(iii)$  There exists a constant $$C_D:= \frac{ I_0 \overline{d_I} }{K\overline{v}-\overline{\delta}+\overline{d_I}}>0$$ such that expected deceased cases $\mathbb{E}[D(k)]\leq C_D$ for all $k$. Therefore, the \textit{average epidemic control problem} is solved. In the next section to follow, we provide an illustrative example using historical COVID-19 data to demonstrate our epidemic control performance.

\section{Examples Using Historical Data:  United State}\label{Section: Examples Using Historical Data: Taiwan and United State}
We now illustrate the application of our control methodology on the epidemiological model  using historical data for year~2020
available in \href{https://ourworldindata.org/coronavirus}{https://ourworldindata.org/coronavirus}, 
which contain the number of daily confirmed cases, denoted by $c(k)$, and daily reported deaths, denoted by~$d(k)$ for $k=0,1,\dots, N-1$ for some fixed integer $N$. 

\subsection{A Simple Data-Driven Estimations}
To study the epidemic control  performance,  there are various way to estimate the uncertain parameters; e.g., one can consult ``standard" approach such as minimizing the least-square estimation error to obtain the ``optimal" parameters $\widehat{\delta}$, $\widehat{d}_I$ and $\widehat{v}$; e.g., see~\cite{calafiore2020modified, burke2020data, chowell2017fitting}. However, for the sake of simplicity, 
we now provide a simple ``mean-replacing" approach to estimate the uncertain rate of infected cases~$\delta(\cdot)$ and the death rate $d_I$ via the available data of confirmed cases and the number of reported deaths.
That is,
we begin by recalling that the deceased cases  satisfy $D(k+1)=D(k) + d_I(k) I(k)$. For $k=0,1,\dots, N-1$, given the confirmed (infected) cases $c(k)$ and~reported deaths $d(k)$ and taking $c(k):=I(k)$ and $d(k):=D(k)$, we obtain
\[
d(k+1)=d(k) + d_I(k) c(k).
\]
Then,  the estimate of death rate $d_I(k)$, call it $\widehat{d}_I$, is defined by\footnote{Of course, one can also consult an \textit{moving average} flavored estimate on the $d_I(k)$. Here, as a preliminary work, the simple ``mean-replacing" estimate is used.}
\begin{align}\label{eq: d_I_hat}
	\widehat{d}_I := \frac{1}{N}\sum_{k = 0}^{N-1}  \frac{d(k+1)-d(k)}{c(k)}.
\end{align}
Having found $\widehat{d}_I$, we can now estimate the remaining  two uncertain quantities $\delta(k)$ and $v(k)$. 
To this end, assume that the adopted epidemic control policy  is linear of the form $u(k) = K I(k)$ with~$K :=1$; i.e., we ideally assumed that  government had put all available resources  to control the pandemic.
Now, setting $c(k):=I(k)$, $\widehat{d}_I:=d_I(k)$, and replacing $v(k)$ by its mean $\overline{v}$, we obtain
\begin{align}\label{eq: estimating P(k)}
	c(k+1) = (1+\delta(k)-\widehat{d}_I-\overline{v} )c(k) 
\end{align}
With the aid of equation~(\ref{eq: estimating P(k)}), the estimate of $\delta(k)$, call it $\widehat{\delta}$, is given by
\begin{align}\label{eq: delta_hat}
	\widehat{\delta} : = \frac{1}{N}\sum_{k=0}^{N-1} \left( \frac{c(k+1)}{c(k)}-1  +\widehat{d}_I + \overline{v}\right). 
\end{align}
Having obtained the estimates $\widehat{\delta}$ and $\widehat{d}_I$, we then ready to  apply the epidemiological model and use it to compare it with the available data; see next subsections to follow.

\subsection{Control of COVID-19 Contagion in Almost Sure Sense}
Using the data from the start of the epidemy, over a horizon starting from
March 1,~2020 to September 8,~2020, 
we obtain the estimates $\widehat{\delta}$ with $\delta_{\max}:=\widehat{\delta} \approx 0.5135$ and \mbox{$\widehat{d}_I :=d_{\max} \approx 0.0449.$} 
In addition, we take $L=2$ and assume that the effectiveness rate of control $v(k)$ follows a uniform distribution with $v_{\min}=0.1$ and $v_{\max}=0.2$; i.e., $10\%$ to $20\%$ effectiveness rate on control.
Observe that
\[
\delta_{\max} > \frac{v_{\min}(1-d_{\max})}{v_{\max}} \approx 0.4775.
\]
Hence, Theorem~\ref{theorem: Linear Controlled Infected Cases} does not apply. However, we can still choose a suboptimal feedback gain discussed in Section~\ref{Remarks on Practical Considerations}; i.e.,
\begin{align*}
	K^* := K=\frac{\delta_{\max}}{ v_{\min}}\cdot 1_{\{ \delta_{\max}\leq L v_{\min}\} } + L \cdot 1_{\{\delta_{\max}> L v_{\min}\} }  =L.
\end{align*}
This means that, if the infected cases can not be well-controlled in almost sure sense, governments should put whatever they can to suppress it; here, we see a $100\cdot L\%$ feedback gain is applied. 
The epidemic control  performance is shown in Figure~\ref{fig:infectedcasesdeathsus} where the black solid lines  depict the confirmed cases $c(k)$ and reported deaths $d(k)$ from  historical COVID-19 data in the United States. 
The other thinner lines with various colors depict the epidemic control performance, in terms of $c(k)$ and $d(k)$, under the linear policy $u(k)=K c(k)$ with $K :=L$. Interestingly, the figure also tells us that if $L=2$ is possible, it may yield, in average, a lower confirmed cases. 

\begin{figure}
	\centering
	\includegraphics[width=0.7\linewidth]{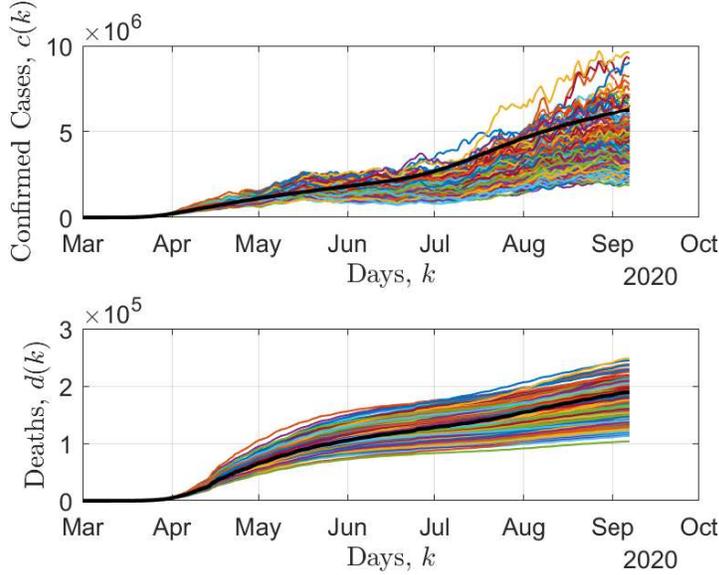}
	\caption{Almost Sure Epidemic Control   Performance:  Confirmed Infected Cases $c(k)$ (top) and Deceased Cases $d(k)$ (bottom).  }
	\label{fig:infectedcasesdeathsus}
\end{figure}


%

\subsubsection{Control of Infected  and Deceased Cases in Expected Value}
With the aid of equations~\ref{eq: d_I_hat} and \ref{eq: delta_hat}, we obtain the estimates~$\widehat{d}_I \approx 0.002$ and \mbox{$\widehat{\delta}:= \overline{\delta} \approx 0.215$.}
Similar to the previous example, we again assume that~$v(k)$ follows a uniform distribution with ${v}_{\min}=0.1$ and $v_{\max}=0.2$; hence the average  $\overline{v} = 0.15$.  
It is readily verified that our estimates satisfy
$$
\overline{\delta} < \frac{(1-d_{\max})\overline{v}}{v_{\max}} \approx 0.998.
$$
Hence, Theorem~\ref{theorem: Linear Controlled on Expected Infected Cases} applies if we take 
\[
K^* := K=\frac{\overline{\delta}}{\overline{v}}\cdot 1_{\{ \overline{\delta} \leq L\overline{v}\} } + L \cdot 1_{\{\overline{\delta} > L \overline{v}\} }  \approx 1.4245.
\]
The control policy above tells us that the government may need  to bring in extra medical resources to achieve $K  > 1$. 
The corresponding epidemic control  performance is shown in Figure~\ref{fig:infectedcasesdeathsuslogexpected2} where the confirmed cases is on the top panel and the deaths is on the bottom panel. 
In the figure, we see a downward trend occurs on the confirmed cases and the saturated deaths as time increases.
The Figure~\ref{fig:infectedcasesdeathsuslogexpected} shows, with y-axis in log-scale, the epidemic control  performance comparison  where the black solid lines are the reported confirmed cases (top) and reported deaths (bottom).

\begin{figure}
	\centering
	\includegraphics[width=0.7\linewidth]{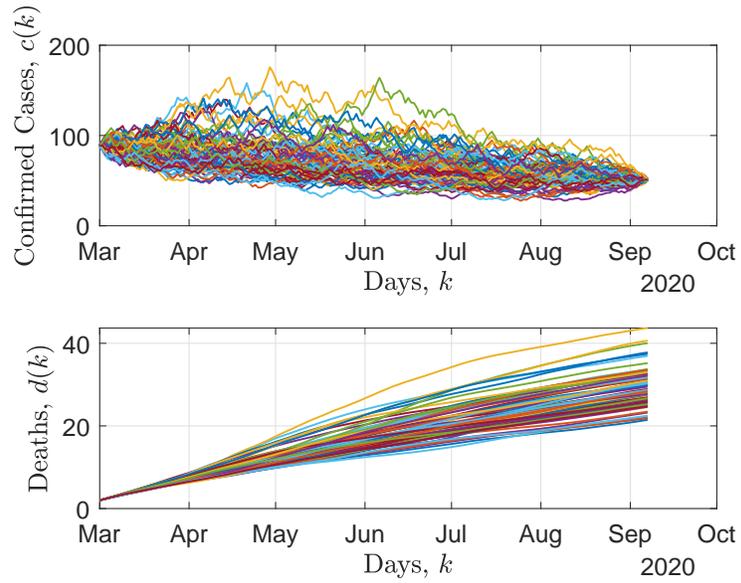}
	\caption{Average Epidemic Control Performance: Confirmed Infected Cases (top) and  Deceased Cases (bottom) under Linear Control Policy $u(k)=K^*I(k)$ with $K^* \approx 1.4245.$}
	\label{fig:infectedcasesdeathsuslogexpected2}
\end{figure}

\begin{figure}
	\centering
	\includegraphics[width=0.7\linewidth]{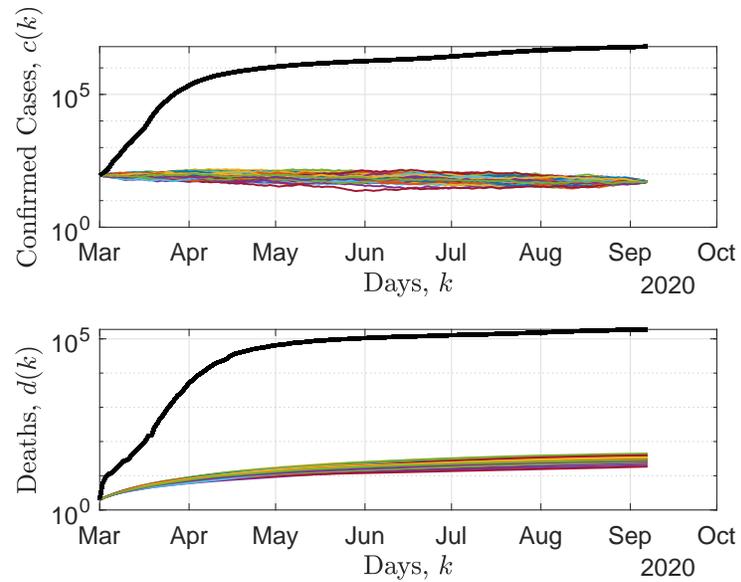}
	\caption{Epidemic Control  Performance Comparison (with y-axis in log-scale). }
	\label{fig:infectedcasesdeathsuslogexpected}
\end{figure}

\section{Concluding Remarks and Future Work} \label{Section: Concluding Remarks}
This preliminary work has been done in the urgency of the ongoing
COVID-19 pandemic, with the mind of providing a simple
yet explainable epidemiological model with a rigorous study on the effectiveness of a linear epidemic control policy class.
Consistent with the existing literature on epidemic modeling and control,
this paper considers a modified stochastic SIRD-based model. 
We analyzed the model and considered two epidemy control problems: One is the almost sure epidemic control  problem and the other is the average epidemic control  problem.
Then,  for both of two problems, with linear control policy, we show  sufficient conditions on feedback gain so that the epidemy is deemed to be  ``well-controlled" in the sense that infected cases goes down to zero asymptotically, and both infected cases and deceased cases are upper bounded uniformly.
Subsequently, we provide a simple data-driven parameter estimations and show some promising numerical results using historical COVID-19 data in the United States. 
Based on our work to date, two important directions immediately present themselves
for future work which described in the next subsections to follow.

\subsection{Generalization to Multi-Population Epidemic Control}
It is possible to extend our analysis to involve multi-population epidemics. To illustrate this, below we consider only the infected cases dynamics. Fix $m$ populations, then for each population $i=1,\dots,m$, we write
\begin{align}\label{eq: infected cases dynamics (Multi_Area)}
	I_i(k+1) &= (1+\delta(k)-d_{I,i}(k))I_i(k)- u_i(k)v_i(k) 
\end{align}
with limited medical resource $$u(k)=\sum_{i=1}^m u_i(k) \leq u_{\max}$$ for some $u_{\max}>0$
and the overall infected cases are 
$$
I(k)=\sum_{i=1}^{m}I_i(k).
$$
It is also possible to consider the case where the control policy is with delay effect; i.e., $u(k-\tilde{d})$ with delay time $\tilde{d}$; see also~\cite{hsieh2019positive} for a discussion on handling the time delay considerations in a class of positive finance systems. 

\subsection{Economic Budgets Consumption} Another possible research direction would be to take the economic budget into play. As seen in Remark~\ref{Remarks on Practical Considerations}, the idea of minimizing the expenditure induced by implementing control policy may lead to a version of ``optimal" choice for feedback gain $K$. In particular, one can even consider the budget dynamics and carry out the optimization. Specifically, let $B(k)$ be the \textit{available medical budget} at the $k$th day which satisfies 
\[
B(k+1) = B(k) + p(k)u(k)
\]
where $p(k)$ is the price to pay per control policy at each stage, which can be modeled as a random variables with finite support. Then,
it is readily verified that 
\[
B(k) =  B_0 + \sum_{i = 0}^{k-1} p(i) u(i)
\]
for $k \geq 1.$ From here, there are many possible directions to pursue. For example, one can consider a optimization problem which minimizes the expected budget cost and infected cases; i.e., one might seek to find a sequence of $K $ which solve
\[
\inf_{K }  \mathbb{E}[\beta B(k) + \gamma I(k)]
\]
for some constants $\beta$ and $\gamma$.

\bibliographystyle{siam} 

\bibliography{refs}
\end{document}